\documentclass[twoside,leqno]{amsart}
\usepackage[T1]{fontenc}
\usepackage[utf8]{inputenc}
\usepackage{microtype}
\usepackage[colorlinks,linkcolor=black,citecolor=black,urlcolor=black, linktocpage=true]{hyperref}
\usepackage{srcltx}
\usepackage[english]{babel}
\numberwithin{equation}{section}
\title[A large sieve inequality]{A large sieve inequality of Elliott-Montgomery-Vaughan type for Maass forms on $GL(n,\mathbb{R})$ with applications}
\author[Y.-K. Lau, M. H. Ng, E. Royer and Y. Wang]{Yuk-Kam Lau, Ming Ho Ng, Emmanuel Royer and Yingnan Wang}
\address[	yklau@maths.hku.hk]{{\sc Yuk-Kam Lau} Department of Mathematics, The University of Hong Kong, Pokfulam Road, Honk Kong P. R. China}
\address[	mhng@math.cuhk.edu.hk]{{\sc Ming Ho Ng} Department of Mathematics, The Chinese University of Hong Kong, Shatin, Honk Kong P. R. China}
\address[	emmanuel.royer@math.cnrs.fr]{{\sc Emmanuel Royer} Université Clermont Auvergne, CNRS, LMBP, F-63000 Clermont-Ferrand, France}
\address[	ynwang@szu.edu.cn]{{\sc Yingnan Wang} Shenzen Key Laboratory of Advanced Machine Learning and Applications, College of Mathematics and Statistics, Shenzen University, Shenzen, Guangdong 518060, P.R. China}
\newtheorem{theorem}{Theorem}[section]
\newtheorem{lemma}{Lemma}[section]
\newtheorem{corollary}{Corollary}[section]
\newtheorem{remark}{Remark}[section]
\newtheorem{proposition}{Proposition}[section]
\newcommand*{\C}{\mathbb{C}}
\newcommand*{\ds}{\displaystyle}

\newcommand*{\ic}{\mathtt{i}}

\newcommand*{\N}{\mathbb{N}}
\newcommand*{\R}{\mathbb{R}}
\newcommand{\ul}{\underline}
 
\newcommand*{\BIBpaper}[1]{{\rm #1.}} 
\newcommand*{\BIBbook}[1]{{\sl #1.}}
\newcommand*{\BIBjournal}[1]{{\sl #1}}

\numberwithin{equation}{section}

\begin{document}
\begin{abstract}
In this paper, we establish a large sieve inequality of Elliott-Montgomery-Vaughan type for Maass forms on $GL(n,\mathbb{R})$ and explore three applications.
\end{abstract}
\maketitle
\begin{center}
First published in Revista Matemática Iberoamericana, 2021, published by the European Mathematical Society. Doi 10.4171/rmi/1238
\end{center}
\section{Introduction}
Elliott \cite{EL}, \cite{EL1}, and Montgomery and Vaughan \cite{M-V} independently developed some sort of large sieve inequalities to study  Linnik's problem, which may yield a more general result than the classical Vinogradov's result, cf. \cite{L-W4}. This device, known as the large sieve inequalities of Elliott-Montgomery-Vaughan (EMV) type, was generalized to the setting of primitive holomorphic cusp forms on $GL(2,\mathbb{R})$ and applied to obtain some statistical results on Hecke eigenvalues of primitive holomorphic cusp forms in \cite{lwu}. Later, Wang \cite{wa} generalized the results to the case of Maass forms on $GL(2,\mathbb{R})$.

It is natural  to ask for a generalization of large sieve inequalities of EMV type to Maass forms on $GL(n,\mathbb{R})\ (n\ge3)$. There are two main difficulties: the first one is that for $n\ge3$ the Hecke relations for $GL(n,\mathbb{R})$ are much more complicated than those of $GL(2,\mathbb{R})$, and the trace formula for $GL(n,\mathbb{R})$ with $n\ge3$ is not as simple as the trace formula (say Kuznetsov's and Petersson's trace formulas) on $GL(2,\mathbb{R})$. Recently, Xiao and Xu \cite{xx}, using Kuznetsov's trace formula and Hecke's relations, made a breakthrough and obtained a large sieve inequality of EMV type to Maass forms on $GL(3,\mathbb{R})$. Moreover, they also applied their large sieve inequality to get a statistical result of sign changes on the Hecke eigenvalues for $GL(3,\mathbb{R})$.

In this paper, we generalize the large sieve inequalities of EMV type to Maass forms on $GL(n,\mathbb{R})$ for all $n\ge3$, and the result is comparable to the case of automorphic forms on $GL(2,\mathbb{R})$ (see \cite{lwu, wa}). Our main tool is the automorphic Plancherel density theorem - a recent great progress due to Matz and Templier \cite{MT}. We remark the  use of properties of (degenerated) Schur's polynomials instead of Hecke's relations to avoid the complicated calculations as  in \cite{xx}. More precisely the (degenerated) Schur polynomial is employed to evaluate the main term when applying the truncated trace formula \cite[Corollary 3.3]{lnw} since the main term in \cite[Corollary 3.3]{lnw} is expressed in the form of orbital integral involving the (degenerated) Schur polynomial by the work of Matz and Templier \cite{MT}. Moreover, we apply our large sieve inequality - Theorem \ref{main theorem} on the $GL(n,\mathbb{R})$ analogue of Linnik's problem, the sign change problems, and the Montgomery-Vaughan conjecture.

Let $\mathcal{H}^\natural=\{\phi_j\}$ be an orthogonal basis consisting of Hecke-Maass cusp forms for $SL(n,\mathbb{R})$. Each $\phi_j$ is associated with a Langlands parameter $\mu_{j}\in \mathfrak{a}_\C^* /W$ where $\mathfrak{a}_\C^* \cong \{ \ul{z}\in \C^n: \sum_i z_i = 0\}$ and $W$ is the Weyl group of $GL(n, \R)$. For $t\ge 1$, we let
\begin{eqnarray}\label{eqHd}
\ds\mathcal{H}_t :=\ds \{ \phi_j\in \mathcal{H}^\natural:  \  \|\mu_{j}\|_2\le t, \ \mu_{j} \in \ic \mathfrak{a}^* \}
\end{eqnarray}
where $\|\cdot\|_2$ is the standard Euclidean norm, and $\ic \mathfrak{a}^*\subset \mathfrak{a}_\C^*$ is isomorphic to $\ic\R^{n-1}$. It is known that $|\mathcal{H}_t|\asymp t^d$ with $d=n(n+1)/2$.

Let $A_{\phi}(m_1,m_2,\ldots,m_{n-1})$ be the Fourier coefficient of $\phi\in\mathcal{H}_t$. In this paper, we normalize each $\phi\in\mathcal{H}_t$ such that
$$A_{\phi}(1,1,\ldots,1)=1.$$
It is well known that
$$A_{\phi}(m_1,m_2,\ldots,m_{n-1})=\overline{A_{\phi}(m_{n-1},m_{n-2},\ldots,m_{1})}.$$
Moreover, for any $\mathbf{\kappa}= (\kappa_1,\cdots, \kappa_{n-1})\in \N_0^{n-1}$ and any prime $p$,
\begin{align}\label{eqApk}
A_\phi(p^{\mathbf{\kappa}})
:= &\ A_\phi(p^{\kappa_1}, p^{\kappa_2}, \cdots, p^{\kappa_{n-1}}) \\
= &\ S_{\mathbf{\kappa}} (\alpha_{\phi,1}(p),\alpha_{\phi,2}(p),\cdots, \alpha_{\phi,n}(p)) \nonumber
\end{align}
where $S_{\mathbf{\kappa}}$ is the (degenerate) Schur polynomial (see Section 2 for definition and refer to \cite{go} or \cite{lw} for a detailed exposition) and $\alpha_\phi(p):= (\alpha_{\phi,1}(p),\alpha_{\phi,2}(p),\cdots, \alpha_{\phi,n}(p))$ is (a representative of) the Satake parameter associated to  $\phi$ at $p$. Every Satake parameter $\alpha_\phi(p)$ satisfies $\prod_{i=1}^n \alpha_{\phi,i}(p)=1$ and
\begin{equation*}
\alpha_{\phi,1}(p)+ \cdots +{\alpha_{\phi,n}(p)}=A_\phi(p,1,\ldots,1).
\end{equation*}
Put $\mathbf{\kappa}^\iota = (\kappa_{n-1}, \cdots, \kappa_1)$ if $\mathbf{\kappa}= (\kappa_1,\cdots, \kappa_{n-1})$. Then we have
\begin{equation*}
A_\phi(p^{\mathbf{\kappa}^\iota}) = A_\phi(p^{\kappa_{n-1}}, \cdots, p^{\kappa_1})=  \overline{A_\phi(p^{\mathbf{\kappa}})},
\end{equation*}
and $A_\phi(p^{\mathbf{\kappa}})\in \R$ if $\mathbf{\kappa}= \mathbf{\kappa}^\iota$.

Notation:  For $\mathbf{\kappa}=(\kappa_1,\ldots,\kappa_{n-1})\in \mathbb{N}_0^{n-1}$, we denote
$\|\mathbf{\kappa}\|:=\sum_{j=1}^{n-1} (n-j) \kappa_j$ and $|\mathbf{\kappa}|=\sum_{j=1}^{n-1}\kappa_j $.

\begin{theorem}\label{main theorem}
Let $0\neq \mathbf{\kappa}=(\kappa_1,\ldots,\kappa_{n-1})\in \mathbb{N}_0^{n-1}$.  Let $j\ge1$ be any integer and let $\left\{b_p\right\}_p$ be a sequence of complex numbers indexed by prime numbers such that $|b_p|\le B$ for some constant $B>0$ and for all primes $p$. Then
\begin{align}\label{large sieve}
&\frac1{|\mathcal{H}_t|}\sum_{\phi
\in\mathcal{H}_t}\left|\sum_{P<p\le Q }b_p\frac{A_\phi(p^{\kappa_1},\ldots,p^{\kappa_{n-1}})}{p}\right|^{2j}
\nonumber\\
&
\ll t^{-1/2}\left(\frac{BC_{\mathbf{\kappa}}Q^{L\|\mathbf{\kappa}\|}}{\log P}\right)^{2j}+
\left(\frac{(BC_{\mathbf{\kappa}})^2 j}{P\log P}\right)^j\left\{1+\left(\frac{40j\log P}{P}\right)^{j/3}\right\}
\end{align}
holds uniformly for
$$
B>0,\qquad j\ge1,\qquad 2\le P< Q\le 2P,
$$
where $L$ is a positive constant, $1\le C_{\mathbf{\kappa}}:=10 (1+|\mathbf{\kappa}|)^{n^2-n}$
and the implied constant depends on $\mathbf{\kappa}$ only.
\end{theorem}

Let $q\ge 2$ be an integer and let $\chi$ be a non principal Dirichlet character modulo $q$. Then the
evaluation of the least integer $n_\chi$ among all positive integers $n$ for which $\chi(n)\neq 0,1$ is referred as Linnik's problem. One generalization formulated to Maass forms on $GL(n,\mathbb{R})$ is the evaluation of the smallest integer $n$ for which
$$
A_{\phi_1}(n,1,\ldots,1)\neq A_{\phi_2}(n,1,\ldots,1),
$$
where
$\phi_{1}\neq \phi_{2}$. We denote this smallest integer by $n_{1,2}$.
The first application uses Theorem \ref{main theorem} to investigate an analogue of Linnik's problem.

 Suppose
$\mathcal{P}$ is a set of prime numbers of positive density in the sense that
\begin{equation}\label{condition}
\sum\limits_{\substack{z<p\le2z \\p\in\mathcal{P}}}\frac{1}{p}\ge\frac{\Delta}{\log z}\ \ \ (\forall \ z\ge z_0),
\end{equation}
with some fixed constants $\Delta>0$ and $z_0>0$.

\begin{theorem}\label{Linnik}
Let $0\neq \mathbf{\kappa}=(\kappa_1,\ldots,\kappa_{n-1})\in \mathbb{N}_0^{n-1}$ and assume the set $\mathcal{P}$ (of primes) satisfies \eqref{condition}.
Let $\Lambda=\left\{\lambda(p)\right\}_p$ be a fixed complex sequence indexed by prime numbers.  For any $\delta>0$,  there is a positive constant $C=C(\delta,  \mathbf{\kappa}, \mathcal{P})$ such that the number of $\phi\in \mathcal{H}_t$ satisfying
$$
A_\phi(p^{\kappa_1},\ldots,p^{\kappa_{n-1}})=\lambda(p)\ \ \ \textrm{for}\ \ \ p\in\mathcal{P},\ \ \ \textrm{and}\ \ \ \delta\log t<p\le 2\delta\log t
$$
is bounded by
$$
\ll t^de^{-C{\log t/\log_2t}}
$$
where $\log_r$ is the $r$-fold iterated logarithm. The implied constant depends at most on $\delta,  \mathbf{\kappa}$ and $\mathcal{P}$.
\end{theorem}
\begin{remark}
Refer to \cite{lwu} and \cite{wa} for the case of $GL(2,\mathbb{R})$.
\end{remark}
\begin{corollary} Let $\phi_{0}\in\mathcal{H}_t$ be fixed and let $\mathcal{P}$ be as stated in Theorem \ref{Linnik}. Let $\ell\in \N$ and let $\delta>0$ be any number.
Then there is a positive constant $C=C(\delta,\ell,\mathcal{P})$ such that the number of $\phi\in \mathcal{H}_t$ satisfying
$$
A_{\phi}(p^\ell,1,\ldots,1)=A_{\phi_{0}}(p^\ell,1,\ldots,1)\ \ \ \textrm{for}\ \ \ p\in\mathcal{P},\ \ \ \textrm{and}\ \ \ \delta\log t<p\le 2\delta\log t
$$
is bounded by
$$
\ll_{\delta,\ell,\mathcal{P}} t^de^{-C{\log t/\log_2t}}.
$$
\end{corollary}
By the corollary, we see that for any fixed $\phi_{1}$, the number of $\phi_{2}\in \mathcal{H}_t$ for which
$$n_{1,2}\ll \log t$$
does not hold is
$$
\ll \big|\mathcal{H}_t\big| e^{-C{\log t/\log_2t}}.
$$

The second application concerns the sign changes of Maass forms on $GL(n,\mathbb{R})$. In the case of $GL(2,\mathbb{R})$, there are fruitful results (for example, see \cite{klsw}, \cite{ma1}, \cite{ma2}). In the case of $GL(3,\mathbb{R})$, Steiger \cite{st} proved that there is a positive proportion of Hecke-Maass forms $\phi$ with positive real part of $A_\phi(p,1)$ for a fixed prime $p$ and Xiao and Xu \cite{xx} gave a statistical result on the signs of $A_\phi(p^{\kappa_1},p^{\kappa_2})+A_\phi(p^{\kappa_2},p^{\kappa_1})$. Applying Theorem \ref{main theorem}, we obtain the following result.
\begin{theorem}\label{sign change}
Let $0\neq \mathbf{\kappa}=(\kappa_1,\ldots,\kappa_{n-1})\in \mathbb{N}_0^{n-1}$.
Let $\left\{\varepsilon_p\right\}_{p\in\mathcal{P}}$ be a sequence of real numbers with $\varepsilon_p\in\left\{\pm1\right\}$ where  the set of primes $\mathcal{P}$ satisfies  \eqref{condition}. For any $\delta>0$, there is a positive constant $C=C(\delta,\mathbf{\kappa},\mathcal{P})$  such that the number of $\phi\in \mathcal{H}_t$ satisfying
$$
\varepsilon_p(A_\phi(p^{\kappa_1},\ldots,p^{\kappa_{n-1}})
+A_\phi(p^{\kappa_{n-1}},\ldots,p^{\kappa_{1}}))>0
$$
for $p\in\mathcal{P}$ and $\delta\log t<p\le2\delta\log t$
is bounded by
$$
\ll t^de^{-C{\log t/\log_2t}}.
$$
The implied constant depends at most on $\delta,  \mathbf{\kappa}$ and $\mathcal{P}$.
\end{theorem}
\begin{remark}
Refer to \cite{lwu} and \cite{wa} for the case of $GL(2,\mathbb{R})$ and to \cite{xx} for $GL(3,\mathbb{R})$.
\end{remark}

The size of $L(1, f)$ for $L$-functions over a family of $f$ has attracted much interest. For $\phi\in\mathcal{H}_t$, its associated $L$-function is defined as
\begin{eqnarray*}
L(s,\phi) := \sum_{m\ge 1} A_\phi(m,1,\cdots,1) m^{-s},
\end{eqnarray*}
for $\Re{\rm e}\, s >(n+1)/2$, and factors into the Euler product
$$
L(s,\phi) =\prod_p \prod_{i=1}^n (1-\alpha_{\phi,i}(p) p^{-s})^{-1} \qquad
$$
where $\alpha_{\phi,i}(p)$, $1\le i\le n$, are the Satake parameters. It is well known that $L(s,\phi)$ can be analytically continued to the whole complex plane.

Recently, Lau and Wang \cite{lw} proved that for all $\phi\in\mathcal{H}_t$, we have
$$
\{1+o(1)\} (2B_n^-\log_2t)^{-A_n^-} \le |L(1,\phi)|\le \{1+o(1)\} (2B_n^+\log_2t)^{A_n^+}.
$$
under the Generalized Ramanujan Conjecture and the Generalized Riemann Hypothesis. Here $B_n^\pm$ are the positive constants in \cite[Lemma~5.3]{lw} and
\begin{equation*}
A_n^+ :=n\qquad \mbox{ and }\qquad A_n^- :=\left\{\begin{array}{ll}
n & \mbox{ if $n$ is even},\vspace{1mm}\\
n\cos(\pi/n) & \mbox{ if $n$ is odd}.
\end{array}
\right.
\end{equation*}
On the other hand, Lau and Wang \cite{lw} also proved that there exist $\phi^\pm\in\mathcal{H}_t$ such that
\begin{equation*}
|L(1,\phi^-)|\le \{1+o(1)\} (B_n^-\log_2t)^{-A_n^-}
\ \mbox{ , } \
|L(1,\phi^+)|\ge \{1+o(1)\} (B_n^+\log_2t)^{A_n^+}.
\end{equation*}
The proportion of such exceptional $\phi^\pm$ in $\mathcal{H}_t$ is at least $ \exp\big(-(\log t)/(\log_2 t)^{3+o(1)}\big)$.
In fact, alongside the Montgomery-Vaughan conjecture (cf. Conjecture 1 in \cite{GS}), the proportion of $\phi^\pm$  in $\mathcal{H}_T$ satisfying $|L(1,\phi^\pm)|^{\pm 1} \ge (B_n^\pm\log_2T)^{A_n^\pm}$ is predicted to
be $>\exp(-C \log t/\log_2t)$ and $<\exp(-c\log t/\log_2t)$ respectively for some constants $C>c>0$.

Theorem \ref{main theorem} gives an  upper bound towards the Montgomery-Vaughan conjecture.
Define
$$
F_{t}^+(s)=\frac{1}{|\mathcal{H}_t|}\sum\limits_{\substack{\phi\in \mathcal{H}_t\\ |L(1,\phi)|>(B_n^+s)^{A_n^+}}}1
$$
and
$$
F_{t}^-(s)=\frac{1}{|\mathcal{H}_t|}\sum\limits_{\substack{\phi\in \mathcal{H}_t\\ |L(1,\phi)|<(B_n^-s)^{A_n^-}}}1.
$$
\begin{theorem}\label{M-V conjecture}
For any $\varepsilon>0$, there are two positive constants $c=c(\varepsilon)$ and $t_0=t_0(\varepsilon)$ such that
$$
F_{t}^\pm(\log_2t+r)\le\exp\left(-c(|r|+1){\frac{\log t}{(\log_2t)(\log_3t)(\log_4t)}}\right)
$$
for $t\ge t_0$ and $\log \varepsilon\le r\le(9-\varepsilon)\log_2t$.
\end{theorem}
\begin{remark}
Refer to \cite{lwu} and \cite{wa} for the case of $GL(2,\mathbb{R})$.
\end{remark}

\section{Preliminaries}

The Fourier coefficients $A_\phi(p^{\mathbf{\kappa}})$ can be expressed in terms of the (degenerate) Schur polynomials and Satake parameters as in \eqref{eqApk}. The degenerate Schur polynomial is defined as
\begin{equation}\label{eq-schur}
S_{\mathbf{\kappa}} (x_{1}, x_{2},\cdots, x_{n}) := \frac{\det \begin{pmatrix} x_j^{\sum_{l=1}^{n-i} (\kappa_l +1)}\end{pmatrix}_{1\le i,j\le n}}{\det \begin{pmatrix} x_j^{\sum_{l=1}^{n-i} 1}\end{pmatrix}_{1\le i,j\le n}}
\end{equation}
for $\mathbf{\kappa}=(\kappa_1,\cdots,\kappa_{n-1})\in \N_0^{n-1}$. Matz and Templier established an automorphic equidistribution of the family $\{A_\phi(p^{\mathbf{\kappa}}): \phi\in \mathcal{H}^\natural\}$ -- the vertical Sato-Tate law for Hecke-Maass forms. Now we explain a consequence of the equidistribution result.

Let $\mathfrak{S}_n$ be the symmetric group and let
$$T_0=\left\{(e^{i\theta_1},e^{i\theta_2},\ldots,e^{i\theta_n})\in (S^{1})^n:e^{i(\theta_1+\theta_2+\cdots+\theta_n)}=1\right\}.$$
We define two measures $d\mu_{\mathrm{ST}}$ and $d\mu_p$ on $T_0/\mathfrak{S}_n$ whose integration formulas (over $[0,2\pi]^{n-1}$) are given by
$$
d\mu_{\mathrm{ST}}=\frac1{n!}\frac1{(2\pi)^{n-1}}\prod_{1\le i<j\le n} |e^{\ic\theta_i}-e^{\ic\theta_j}|^2\, d\theta_1\cdots d\theta_{n-1}
$$
and
$$
d\mu_p=\frac1{n!} \prod_{i=2}^n \frac{1-p^{-i}}{1-p^{-1}}  \cdot \prod_{1\le i< j\le n} \left|\frac{e^{\ic \theta_i}- p^{-1}e^{\ic\theta_j}}{e^{\ic \theta_i}- e^{\ic\theta_j}}\right|^{-2} \cdot \frac1{(2\pi)^{n-1}}\, d\theta_1\cdots {d} \theta_{n-1}.
$$

Define $S_{\mathbf{\kappa}}(1,\cdots,1)$ by taking $x_i\to 1$.
By \cite[Lemma 7.1 (2)]{lw}, we have for any $X\ge 1$ and $\mathbf{\kappa} \in \N_0^{n-1}$,
\begin{equation}\label{new eq 1}
\max_{|x_i|\le X,\, \forall \, i} |S_{\mathbf{\kappa}}(x_1,\cdots, x_n)|\le X^{\|\mathbf{\kappa}\|} S_{\mathbf{\kappa}}(1,\cdots, 1)\le X^{\|\mathbf{\kappa}\|} (1+|\mathbf{\kappa}|)^{n^2-n}.
\end{equation}
A consequence of Matz and Templier's work on the vertical Sato-Tate is the following,  cf. \cite[Corollary 3.3]{lnw}.
\begin{lemma}\label{trace formula}  Let $\mathbf{\kappa}=(\kappa_1,\cdots,\kappa_{n-1})\in \N_0^{n-1}$, $\mathcal{H}_t$ and $A_\phi(p^\mathbf{\kappa})=A_\phi (p^{\kappa_1},\cdots,p^{\kappa_{n-1}})$ be defined as above, cf. \eqref{eqHd}, \eqref{eqApk} and \eqref{eq-schur}. Then for any $\ell,m\in \N$,
\begin{align*}
&\frac1{|\mathcal{H}_t|}\sum_{\phi
\in\mathcal{H}_t}\prod_{p^{u_p}\|\ell,p^{v_p}\|m}
A_\phi(p^{\kappa_1},\ldots,p^{\kappa_{n-1}})^{u_p}
\overline{A_\phi(p^{\kappa_1},\ldots,p^{\kappa_{n-1}})}^{v_p}\nonumber\\
&=\prod_{p^{u_p}\|\ell,p^{v_p}\|m}\int_{T_0/\mathfrak{S}_n}
S_{\mathbf{\kappa}}^{u_p}\overline{S_{\mathbf{\kappa}}^{v_p}}d\mu_p
+O\bigg(t^{-1/2}
\prod_{p^{u_p}\|\ell,p^{v_p}\|m} \big(c_{\mathbf{\kappa}}p^{L\|\mathbf{\kappa}\|}\big)^{u_p+v_p}\bigg)
\end{align*}
where
$L$ is a positive constant, $1\le c_{\mathbf{\kappa}}:=
(1+|\mathbf{\kappa}|)^{n^2-n}$.
\end{lemma}

The product of two Schur polynomials $S_{\mathbf{\kappa}}$ and $S_{\mathbf{\kappa}'}$ may be evaluated with the Littlewood-Richardson rule:
\begin{equation}\label{lr}
S_{\mathbf{\kappa}} S_{\mathbf{\kappa}'} =S_{\mathbf{\kappa}} \cdot S_{\mathbf{\kappa}'} = \sum_{\mathbf{\xi}} d_{\mathbf{\kappa}\mathbf{\kappa}'}^{\mathbf{\xi}} S_{\mathbf{\xi}}
\end{equation}
where the $d_{\mathbf{\kappa}\mathbf{\kappa}'}^{\mathbf{\xi}}$'s are nonnegative integers and the summation runs over $\mathbf{\xi}\in \N_0^{n-1}$ satisfying $\|\mathbf{\xi}\|\le \|\mathbf{\kappa}\| + \|\mathbf{\kappa}'\|$ and $\|\mathbf{\xi}\|\equiv \|\mathbf{\kappa}\| + \|\mathbf{\kappa}'\|$ mod $n$. (Recall that $\|\mathbf{\kappa}\|:=\sum_i (n-i)\kappa_i$.)
Moreover $\{S_{\mathbf{\kappa}}\}$ form an orthonormal set under the inner product induced by the measure $d\mu_{\mathrm{ST}}$,
\begin{eqnarray}\label{2.4 in lnw}
\langle { S}_{\mathbf{\kappa}}, { S}_{\mathbf{\kappa}'}\rangle &=&
\int_{[0,2\pi]^{n-1}} { S}_{\mathbf{\kappa}}(\ul{\theta}) \overline{{ S}_{\mathbf{\kappa}'}(\ul{\theta})} d\mu_{\mathrm{ST}}
= \delta_{\mathbf{\kappa}=\mathbf{\kappa}'}.
\end{eqnarray}
As well,  by \cite[Proposition 7.4 (1)]{lw} we have
\begin{align*}
\int_{T_0/\mathfrak{S}_n}
S_{\mathbf{\kappa}}d\mu_p=\prod_{i=1}^{n-1}(1-p^{-i})\cdot
\sum_{\mathbf{\eta}\in\mathbb{N}_0^{n-1}}
d_{\mathbf{\kappa}\mathbf{\eta}}^{\mathbf{\eta}}\cdot p^{-\|\mathbf{\eta}\|}
\end{align*}
where the sum over $\mathbf{\eta}$ is supported on $|\mathbf{\eta}|\ge\|\mathbf{\kappa}\|/n$ and with \eqref{new eq 1} and \eqref{2.4 in lnw},
$$
0\le d_{\mathbf{\kappa}\mathbf{\eta}}^{\mathbf{\eta}}
=\int_{T_0/\mathfrak{S}_n}S_{\mathbf{\kappa}}|S_{\mathbf{\eta}}|^2d\mu_{\mathrm{ST}}
\le(1+|\mathbf{\kappa}|)^{(n^2-n)}.
$$
Consequently, for $\|\kappa\|\neq 0$ we have
\begin{align}\label{squarefree main term}
& \Big|\int_{T_0/\mathfrak{S}_n} S_{\mathbf{\kappa}}d\mu_p\Big|\nonumber\\
&\le (1+|\mathbf{\kappa}|)^{(n^2-n)} \prod_{i=1}^{n-1}(1-p^{-i}) \max_{\sum_i\eta_i = \lceil\frac{\|\kappa\|}n\rceil}\Big( \prod_{1\le i\le n-1}\sum_{\ell\ge \eta_i} p^{-i \ell}\Big)\nonumber\\
&\le (1+|\mathbf{\kappa}|)^{(n^2-n)}\max_{|\eta| = \lceil\frac{\|\kappa\|}n\rceil}p^{-\|\eta\|}\nonumber\\
&\le (1+|\mathbf{\kappa}|)^{(n^2-n)}p^{-1}
\end{align}
where $\lceil x\rceil$ denotes the smallest integer greater than or equal to $x$. (Note $|\eta|\le \|\eta\|$.)

By Cauchy-Schwarz's inequality and \eqref{new eq 1}, we have
\begin{align}\label{new eq}
 \sum_{\|\mathbf{\xi}\|\le n|\mathbf{\kappa}|}
(d_{\mathbf{\kappa}\mathbf{\kappa}'}^{\mathbf{\xi}})^2
&=\langle S_{\mathbf{\kappa}}S_{\mathbf{\kappa}'},S_{\mathbf{\kappa}}S_{\mathbf{\kappa}'} \rangle\nonumber\\
&\le S_{\mathbf{\kappa}}(1,\ldots,1)S_{\mathbf{\kappa}'}(1,\ldots,1)
\langle S_{\mathbf{\kappa}},S_{\mathbf{\kappa}}\rangle^{1/2}
\langle S_{\mathbf{\kappa}'},S_{\mathbf{\kappa}'}\rangle^{1/2}\nonumber\\
&\le ((1+|\mathbf{\kappa}|)(1+|\mathbf{\kappa}'|))^{n^2-n}= c_{\mathbf{\kappa}}c_{\mathbf{\kappa}'}.
\end{align}

We need an arithmetic function and a result from \cite{lwu}.
\begin{lemma}\label{lem-a}
Let $2\le P< Q\le 2P$, $j\ge 1$ and $n\ge 1$. Define
$$a_j(n)=a_j(n;P,Q)
=|\left\{(p_1,\ldots,p_j):p_1\cdots p_j=n,\ P<p_1,\ldots,p_j\le Q\right\}|.$$
For any $d>0$, $\sum_n a_j(n) d^{\Omega(n)}/n \ll (3d/\log P)^j$; moreover,
\begin{eqnarray*}
\sum_n a_j(n^2) \frac{d^{\Omega(n)}}{n^2} &\le & \delta_{2|j} \Big(\frac{3dj}{P\log P}\Big)^{j/2} \\
{\sum_n}^\natural  a_j(n) \frac{d^{\Omega(n)}}{n} &\le &  \Big(\frac{12d^2j}{P\log P}\Big)^{j/2}\Big\{1+ \Big(\frac{j\log P}{54 P}\Big)^{j/6}\Big\} \\
{\sum_m}^\flat \mathop{{\sum}^\natural}_{(m,n)=1}  a_j(mn) \frac{d^{\Omega(mn)}}{m^2n} &\le &  \Big(\frac{48d^2j}{P\log P}\Big)^{j/2}\Big\{1+ \Big(\frac{20 j\log P}{P}\Big)^{j/6}\Big\}
\end{eqnarray*}
where $\Omega(n)$ counts the number of (not necessarily distinct) prime divisors, $\delta_{2|j}=1$ if $2|j$ or $0$ otherwise, $\sum^\flat$ and  $\sum^\sharp$ run over squarefree and squarefull integers respectively.
\end{lemma}

\section{Proof of Theorem \ref{main theorem}}

Let $a_j(\cdot)$ be defined as in Lemma~\ref{lem-a}. Squaring out, we have
\begin{align*}
&\left|\sum_{P<p\le Q }b_p\frac{A_\phi(p^{\kappa_1},\ldots,p^{\kappa_{n-1}})}{p}\right|^{2j}\nonumber\\
&=\left(\sum_{P<p\le Q }b_p\frac{A_\phi(p^{\kappa_1},\ldots,p^{\kappa_{n-1}})}{p}\right)^j
\left(\sum_{P<p\le Q }\overline{b_p}\frac{\overline{A_\phi(p^{\kappa_1},\ldots,p^{\kappa_{n-1}})}}{p}\right)^j\nonumber\\
&=\left(\sum_{P^j<\ell\le Q^j}a_j(\ell)\frac{b_\ell}{\ell}\prod_{p^u\|\ell}
A_\phi(p^{\kappa_1},\ldots,p^{\kappa_{n-1}})^u\right) \nonumber\\
& \mbox{ }\phantom{===} \times
\left(\sum_{P^j<m\le Q^j}a_j(m)\frac{\overline{b_m}}{m}\prod_{q^v\|m}
\overline{A_\phi(p^{\kappa_1},\ldots,p^{\kappa_{n-1}})}^v\right)\nonumber\\
&=\sum_{P^j<\ell,m\le Q^j}a_j(\ell)a_j(m)\frac{b_\ell\overline{b_m}}{\ell m}\\
& \mbox{ }\phantom{===} \times
\prod_{p^{u_p}\|\ell,p^{v_p}\|m}
A_\phi(p^{\kappa_1},\ldots,p^{\kappa_{n-1}})^{u_p}
\overline{A_\phi(p^{\kappa_1},\ldots,p^{\kappa_{n-1}})}^{v_p}
\end{align*}
Averaging over $\phi\in \mathcal{H}_t$, it follows from Lemma~\ref{trace formula} that
\begin{align}\label{trace}
& \frac1{|\mathcal{H}_t|}\sum_{\phi
\in\mathcal{H}_t}\prod_{p^{u_p}\|\ell,p^{v_p}\|m} \cdots \\
&=\prod_{p^{u_p}\|\ell,p^{v_p}\|m}\int_{T_0/\mathfrak{S}_n}
S_{\mathbf{\kappa}}^{u_p}\overline{S_{\mathbf{\kappa}}^{v_p}}d\mu_p
+O\left(t^{-1/2}
\big(c_{\mathbf{\kappa}}Q^{L\|\mathbf{\kappa}\|}\big)^{2j}\right).\nonumber
\end{align}
Thus the left side of \eqref{large sieve} can be expressed as follows:
\begin{align}
\frac1{|\mathcal{H}_t|}\sum_{\phi
\in\mathcal{H}_t}\left|\sum_{P<p\le Q }b_p\frac{A_\phi(p^{\kappa_1},\ldots,p^{\kappa_{n-1}})}{p}\right|^{2j}
= M + E.
\end{align}
The error term $E$ is
\begin{align}\label{contributio of error}
&\ll t^{-1/2}
\big(c_{\mathbf{\kappa}}Q^{L\|\mathbf{\kappa}\|}\big)^{2j}\sum_{P^j<\ell,m\le Q^j}a_j(\ell)a_j(m)\frac{b_\ell\overline{b_m}}{\ell m}\nonumber\\
&\ll t^{-1/2}
\big(c_{\mathbf{\kappa}}Q^{L\|\mathbf{\kappa}\|}\big)^{2j} \left(\frac{3B}{\log P}\right)^{2j}
\end{align}
by Lemma~\ref{lem-a}.

Next we evaluate the main term
\begin{align}\label{main term}
M& = \sum_{P^j<\ell,m\le Q^j}a_j(\ell)a_j(m)\frac{b_\ell\overline{b_m}}{\ell m}\prod_{p^{u_p}\|\ell,p^{v_p}\|m}\int_{T_0/\mathfrak{S}_n}
S_{\mathbf{\kappa}}^{u_p}\overline{S_{\mathbf{\kappa}}^{v_p}}d\mu_p.
\end{align}
Write $\ell =\ell_1\ell'$ and $m=m_1m'$ such that $\ell_1m_1$ is squarefree, $\ell'm'$ is squarefull and $(\ell_1m_1,\ell'm')=1$.\footnote{The decomposition is unique. Assume  $\ell =\ell_1\ell'=\ell_2\ell''$ and $m=m_1m'=m_2m''$ are two such decomposition.  Every positive integer  decomposes uniquely into a product of a squarefree integer and a squarefull integer.  From $(\ell_1m_1)(\ell'm')=(\ell_2m_2)(\ell'' m'')$,  we get $(*)$: $\ell_1m_1=\ell_2m_2$ and $\ell'm'=\ell''m''$. As $\ell_1m_1$ is squarefree, we have $(\ell_1,m_1)=1$; with $(\ell_1m_1,\ell'm')=1$, we infer $(\ell_1,m)=1$. So $(\ell_1,m_2)=1$, and $(\ell_2, m_1)=1$ by symmetry. By $(*)$, $\ell_1=\ell_2$ and $m_1=m_2$.}  (Note $\ell_1m_1=1$ when $\ell m$ is squarefull.) Set $h=\ell_1m_1$ and $r= \ell' m'$.  We split the product over prime divisors of $\ell m$ in \eqref{main term} into a product of two pieces over prime divisors of $\ell_1m_1$ and $\ell'm'$ respectively:
\begin{align*}
\prod_{p^{u_p}\|\ell,p^{v_p}\|m}\cdots
=
\prod_{p^{u_p}\|\ell_1,p^{v_p}\|m_1}\int_{T_0/\mathfrak{S}_n}
S_{\mathbf{\kappa}}^{u_p}\overline{S_{\mathbf{\kappa}}^{v_p}}d\mu_p
\prod_{p^{u_p}\|\ell',p^{v_p}\|m'}\int_{T_0/\mathfrak{S}_n}
S_{\mathbf{\kappa}}^{u_p}\overline{S_{\mathbf{\kappa}}^{v_p}}d\mu_p.
\end{align*}
Inside the second product, we invoke the trivial bound \eqref{new eq 1} and for the first product, (as $\ell_1m_1$ is squarefree) we have $u_p+v_p=1$ and thus apply \eqref{squarefree main term}. This leads to
\begin{align*}
&\Big|\prod_{p^{u_p}\|\ell,p^{v_p}\|m}\int_{T_0/\mathfrak{S}_n}
S_{\mathbf{\kappa}}^{u_p}\overline{S_{\mathbf{\kappa}}^{v_p}}d\mu_p\Big|\\
& \le
(1+|\mathbf{\kappa}|)^{\Omega(\ell'm') (n^2-n)} \prod_{p^{u_p}\|\ell_1,p^{v_p}\|m_1}(1+|\mathbf{\kappa}|)^{n^2-n}p^{-1} \\
&\le (1+|\mathbf{\kappa}|)^{2j(n^2-n)} h^{-1},
\end{align*}
and
\begin{align*}
|M|
&\le (1+|\mathbf{\kappa}|)^{2j(n^2-n)} \sum_{P^j<\ell_1\ell',m_1m'\le Q^j}a_j(\ell_1\ell')a_j(m_1m')\frac{\big|b_{\ell_1\ell'}\overline{b_{m_1m'}}\big|}{(\ell_1m_1)^2\ell' m'}
\nonumber\\
&\le (1+|\mathbf{\kappa}|)^{2j(n^2-n)}B^{2j}\sideset{}{^\flat}\sum_{h}\sideset{}{^\natural}\sum_{r}\frac1{h^2r} \sum_{\substack{P^j<\ell_1\ell',m_1m'\le Q^j\\ \ell_1 m_1= h, \, \ell'm'=r}} a_j(\ell_1\ell')a_j(m_1m')
\nonumber\\
&\le (1+|\mathbf{\kappa}|)^{2j(n^2-n)}B^{2j}\sideset{}{^\flat}\sum_{h}\sideset{}{^\natural}\sum_{r}\frac{a_{2j}(hr)}{h^2r}\nonumber\\
&\ll (1+|\mathbf{\kappa}|)^{2j(n^2-n)}B^{2j}
\left(\frac{96j}{P\log P}\right)^j\left\{1+\left(\frac{40j\log P}{P}\right)^{j/3}\right\}
\end{align*}
where the implied constant is independent of $j$.

\section{Proof of Theorem \ref{Linnik}}

 Let $\delta \log t\le P\le(\log t)^{10}$ and write $ \mathcal{P}_P:= \mathcal{P}\cap (P, 2P]$. Define
$$
E(t;{P})=\left\{\phi\in \mathcal{H}_t:A_\phi(p^{\kappa_1},\ldots,p^{\kappa_{n-1}})=\lambda(p)\ \textrm{for}\
 p\in\mathcal{P}\cap (P, 2P]\right\}.
$$
As the Ramanujan Conjecture is open, we consider the exceptional set over each prime
$$
\mathcal{E}(t,p)=\left\{\phi\in \mathcal{H}_t:\log\max_{1\le i\le n}|\alpha_{\phi,i}(p)|>1\right\}
$$
whose size is under control. Indeed, analogously to Sarnak's bound for the $GL(2)$ Maass forms, we have $|\mathcal{E}(t,p)|\ll t^{d-c_0/\log p}$ where $c_0>0$ is a constant, cf. \cite[Theorem 7.3]{lw}. Hence
\begin{equation*}
\Big|\bigcup_{p\in\mathcal{P}_P}\mathcal{E}(t,p)\Big|\ll t^{d-c'/\log P}
\end{equation*}
for some constant $c'$. Set
$$
E^*(t;{P})=E (t;{P})\setminus\bigcup_{p\in\mathcal{P}_P}\mathcal{E}(t,p).
$$

It remains to prove that
$$
E^*(t;{P})\ll_{\delta,\mathbf{\kappa},\mathcal{P}} t^de^{-C{\log t/\log_2t}}
$$
for all $t>T_0$, where $T_0=T_0(\delta, \mathbf{\kappa},\mathcal{P})$ is a sufficiently large number.
We may assume
\begin{equation}\label{bound of lambda p}
|\lambda(p)|< e^{\|\mathbf{\kappa}\|}(1+|\mathbf{\kappa}|)^{n^2-n}
\end{equation}
for all $P\le p\le 2P$; otherwise the set $E(t;{P})$ is empty by \eqref{new eq 1}. Suppose $j\in \N$ is chosen such that
\begin{align}\label{j}
j \le \frac{P}{40 \log P}.
\end{align}
We apply Theorem \ref{main theorem} with
\begin{equation}\label{bp}
b_p=\begin{cases}
         \overline{\lambda(p)} &\textrm{if\ }p\in\mathcal{P}_P,\\
                  0 &\textrm{otherwise}.
                          \end{cases}
\end{equation}
Since $\overline{\lambda(p)} A_\phi(p^{\kappa_1},\ldots,p^{\kappa_{n-1}}) = |A_\phi(p^{\kappa_1},\ldots,p^{\kappa_{n-1}})|^2$ for $\phi \in E^*(t;{P})$, it follows that
\begin{align}\label{5}
&\sum\limits_{\phi\in E^*(t;{P})}\bigg|\sum\limits_{ p\in\mathcal{P}_P}\frac{|A_\phi(p^{\kappa_1},\ldots,p^{\kappa_{n-1}})|^2}{p}\bigg|^{2j} \\
&\le\sum\limits_{\phi\in \mathcal{H}_t}\Big|\sum\limits_{P<p\le2P}b_p
\frac{A_\phi(p^{\kappa_1},\ldots,p^{\kappa_{n-1}})}{p}\Big|^{2j}\nonumber\\
&\ll t^d\left(\frac{(B_1C_{\mathbf{\kappa}})^2 j}{P\log P}\right)^j
+t^{d-1/2}\left(\frac{B_1C_{\mathbf{\kappa}}Q^{L\|\mathbf{\kappa}\|}}{\log P}\right)^{2j}\nonumber
\end{align}
where $B_1=e^{\|\mathbf{\kappa}\|}(1+|\mathbf{\kappa}|)^{n^2-n}$ and $Q=2P$, in view of \eqref{bound of lambda p}.

The size of  $|A_\phi(p^{\kappa_1},\ldots,p^{\kappa_{n-1}})|^2$  is about $1$ on average. To see it, we firstly deduce from \eqref{eqApk} and \eqref{lr} that
\begin{align}\label{key1}
 |A_\phi(p^{\kappa_1},\ldots,p^{\kappa_{n-1}})|^2
&=A_\phi(p^{\kappa_1},\ldots,p^{\kappa_{n-1}})
{A_\phi(p^{\kappa_{n-1}},\ldots,p^{\kappa_{1}})}\nonumber\\
&=1+\sum_{\substack{\mathbf{\xi}\neq\mathbf{0}\\\|\mathbf{\xi}\|\le n|\mathbf{\kappa}|}}
d_{\mathbf{\kappa}\mathbf{\kappa}^\iota}^{\mathbf{\xi}}A_\phi(p^{\xi_1},\ldots,p^{\xi_{n-1}})
\end{align}
where
$\mathbf{\kappa}^\iota=(\kappa_{n-1},\ldots,\kappa_1)$. (Then $\|\mathbf{\kappa}^\iota\|=n|\mathbf{\kappa}|-\|\mathbf{\kappa}\|$.)

Secondly, we exploit the oscillation among $A_\phi(p^{\xi_1},\ldots,p^{\xi_{n-1}})$ by Theorem~\ref{main theorem} (again).
For $\mathbf{\xi}=(\xi_1,\ldots,\xi_{n-1})$ with $1\le \|\mathbf{\xi}\|\le n|\mathbf{\kappa}|=|n\mathbf{\kappa}|$, we define
$$
E^{\mathbf{\xi}}(t;{P})=\bigg\{\phi\in \mathcal{H}_t:\bigg|\sum\limits_{\substack{P<p\le2P\\ p\in\mathcal{P}}}\frac{A_\phi(p^{\xi_1},\ldots,p^{\xi_{n-1}})}{p}\bigg|
\ge\frac{\Delta'}{\log P}\bigg\}
$$
where $\Delta' := \Delta/(2 c_{\mathbf{\kappa}} c_{\mathbf{\kappa}^\iota})< \Delta/2$.  Taking $b_p=  1$  if $p\in\mathcal{P}_P$ or $0$  otherwise,
we get from Theorem \ref{main theorem} with $C_\xi\le C_{n\kappa}$ that
\begin{align}\label{6}
|E^{\mathbf{\xi}}(t;{P})| & \ll t^d\left(\frac{C_{n\kappa}^2
j\log P}{{\Delta'}^2P}\right)^j
+t^{d-1/2}\left(\frac{C_{n\kappa} Q^{L\|\mathbf{\xi}\|}}{\Delta'}\right)^{2j}.
\end{align}
For $\phi \in E^*(t;{P})\backslash\bigcup\limits_{\substack{\mathbf{\xi}\neq\mathbf{0}
\\\|\mathbf{\xi}\|\le n|\mathbf{\kappa}|}} E^{\mathbf{\xi}}(t;{P})$, the inner sum (over $p$) in \eqref{5} is, by \eqref{key1},
 \begin{align}\label{7}
 \ge \sum\limits_{\substack{P<p\le 2P\\ p\in\mathcal{P}}}\frac{1}{p}-\sum_{\substack{\mathbf{\xi}\neq\mathbf{0}\\\|\mathbf{\xi}\|\le n|\mathbf{\kappa}|}}
d_{\mathbf{\kappa}\mathbf{\kappa}^\iota}^{\mathbf{\xi}}\bigg|\sum\limits_{\substack{P<p\le 2P\\ p\in\mathcal{P}}}\frac{A_\phi(p^{\xi_1},\ldots,p^{\xi_{n-1}})}{p}\bigg|\ge \frac{\Delta}{2\log P}
\end{align}
Here we have applied that $c_{\mathbf{\kappa}} c_{\mathbf{\kappa}^\iota} \Delta' \le \Delta/2$ and
\begin{align}\label{new eq 2}
\sum_{\substack{\mathbf{\xi}\neq\mathbf{0}\\\|\mathbf{\xi}\|\le n|\mathbf{\kappa}|}}
d_{\mathbf{\kappa}\mathbf{\kappa}^\iota}^{\mathbf{\xi}}
\le \sum_{\|\mathbf{\xi}\|\le n|\mathbf{\kappa}|}
(d_{\mathbf{\kappa}\mathbf{\kappa}^\iota}^{\mathbf{\xi}})^2
\le c_{\mathbf{\kappa}} c_{\mathbf{\kappa}^\iota}
\end{align}
by \eqref{new eq}.

Applying the lower bound \eqref{7} to the left-hand side of (\ref{5}), we thus infer
\begin{align*}
& \left(
\frac{\Delta}{2\log P}\right)^{2j}\Big|E^*(t;{P})\backslash\bigcup\limits_{\substack{\mathbf{\xi}\neq\mathbf{0}
\\\|\mathbf{\xi}\|\le n|\mathbf{\kappa}|}} E^{\mathbf{\xi}}(t;{P})\Big| \\
&\ll t^d\left(\frac{(B_1C_{\mathbf{\kappa}})^2 j}{P\log P}\right)^j
+t^{d-1/2}\left(\frac{B_1C_{\mathbf{\kappa}}Q^{L\|\mathbf{\kappa}\|}}{\log P}\right)^{2j}\nonumber
\end{align*}
and, together with \eqref{6},
\begin{align}\label{8}
|E^*(t;{P})|\ll
t^d\left(\frac{(B_1C_{n\kappa})^2 j\log P}{{\Delta'}^2P}\right)^j
+t^{d-1/2}\left(\frac{B_1 C_{n\kappa} Q^{L\|\mathbf{\kappa}\|}}{\Delta'}\right)^{2j}.
\end{align}
Recall $\delta \log t \le P \le (\log t)^{10}$.  Take
$$
j=\left\lceil\Delta^*{\frac{\log t}{\log P}} \right\rceil
$$
with
$$
\Delta^*=\min\Big(\frac{\delta}{40},  \frac{ \delta \Delta'^2 }{(2 B_1 C_{n\kappa})^2}, \frac1{8L\|\mathbf{\kappa}\|}\Big).
$$

Thus \eqref{j} is valid and the term inside the first bracket of \eqref{8} is bounded by $1/4$. Let $T_0$ be large enough so that $1<j< \delta (\log t)/ (\log_2 t)$ and the second term in the right-side of \eqref{8} is less than $t^{d-1/6}$ whenever $t>T_0$.  Then we conclude that
\begin{align*}
&|E^*(t;{P})|\ll t^de^{-C{\log t/\log_2 t}}
\end{align*}
for some constant $C>0$ depending on $\delta, \mathbf{\kappa}$ and $\mathcal{P}$. The proof  of Theorem \ref{Linnik}
is complete.

\section{Proof of Theorem \ref{sign change}}

The method of proof is the same as Theorem~\ref{Linnik}, starting with the set
\begin{align*}
F (t;{P})=\left\{\phi\in \mathcal{H}_t:\varepsilon_p \big(A_\phi(p^{\kappa_1},\ldots,p^{\kappa_{n-1}}) +A_\phi(p^{\kappa_{n-1}},\ldots,p^{\kappa_{1}})\big)>0 \ \textrm{for}\  p\in\mathcal{P}_P \right\}.
\end{align*}
The task is to evaluate
$$
F^*(t;{P})=F(t;{P})\setminus \bigcup_{p\in\mathcal{P}_P}\mathcal{E}(t,p).
$$
Using the positivity of $\varepsilon_p\big(A_\phi(p^{\kappa_1},\ldots,p^{\kappa_{n-1}})
+A_\phi(p^{\kappa_{n-1}},\ldots,p^{\kappa_{1}})\big)$ for  $\phi\in F^*(t;{P})$, we have
\begin{align*}
&|A_\phi(p^{\kappa_1},\ldots,p^{\kappa_{n-1}})
+A_\phi(p^{\kappa_{n-1}},\ldots,p^{\kappa_{1}})|^2\\
&\le 2e^{\|\mathbf{\kappa}\|}(1+|\mathbf{\kappa}|)^{n^2-n}
\varepsilon_p\big(A_\phi(p^{\kappa_1},\ldots,p^{\kappa_{n-1}})
+A_\phi(p^{\kappa_{n-1}},\ldots,p^{\kappa_{1}})\big)
\end{align*}
by \eqref{new eq 1}, and the analogue of \eqref{key1} follows from \eqref{lr} and \eqref{2.4 in lnw}:
\begin{align*}
&  |A_\phi(p^{\kappa_1},\ldots,p^{\kappa_{n-1}})
+A_\phi(p^{\kappa_{n-1}},\ldots,p^{\kappa_{1}})|^2 \\
&= 2A_\phi(p^{\kappa_1},\ldots,p^{\kappa_{n-1}})
{A_\phi(p^{\kappa_{n-1}},\ldots,p^{\kappa_{1}})} \\
& \mbox{ } \phantom{=} +
A_\phi(p^{\kappa_1},\ldots,p^{\kappa_{n-1}})^2
+A_\phi(p^{\kappa_{n-1}},\ldots,p^{\kappa_{1}})^2\\
&=  2(1+\delta_{\mathbf{\kappa},\mathbf{\kappa}^\iota})+\sum_{\substack{\mathbf{\xi}\neq\mathbf{0}\\|\mathbf{\xi}\|\le 2n|\mathbf{\kappa}|}}
(d_{\mathbf{\kappa}\mathbf{\kappa}}^{\mathbf{\xi}}
+2d_{\mathbf{\kappa}\mathbf{\kappa}^\iota}^{\mathbf{\xi}}
+d_{\mathbf{\kappa}^\iota\mathbf{\kappa}^\iota}^{\mathbf{\xi}})
A_\phi(p^{\xi_1},\ldots,p^{\xi_{n-1}})
\end{align*}
where $\delta_{\mathbf{\kappa},\mathbf{\kappa}^\iota}$ if $\mathbf{\kappa}=\mathbf{\kappa}^\iota$ or $0$ otherwise, and $\mathbf{\kappa}^\iota=(\kappa_{n-1},\ldots,\kappa_1)$.

\section{Proof of Theorem \ref{M-V conjecture}}

Let $\varepsilon \in(0, {10^{-10}}]$ be fixed. We need a short Euler product approximation for a bulk of $L(1,\phi)$'s.
\begin{proposition}\label{proposition1}
There are a constant $c'>0$ and a subset $E^1(z)$ of $\mathcal{H}_t$ such that
\begin{equation*}
L(1,\phi)=\left\{1+O\left(\frac{1}{\log_2 t}\right)\right\}\prod\limits_{p\le z}\prod\limits_{i=1}^n\left(1-\frac{\alpha_{\phi,i}(p)}{p}\right)^{-1}
\end{equation*}
uniformly for $\varepsilon\log t\le z\le(\log t)^{10}$ and all Maass forms $\phi\in \mathcal{H}_t\setminus E^1(z)$,  where the implied constant in the $O$-term is absolute and
$$
|E^1(z)|= O_\varepsilon\left( t^d\exp\left(-c'\frac{\log t}{(\log_2t)(\log_3t)(\log_4t)}\right)\right).
$$
\end{proposition}

\begin{proof} We follow the same approach as in the proof of \cite[Proposition 8.1]{lwu}. A crucial difference is without the Ramanujan bound now, and thus we exclude the forms outside the set
$$
\mathcal{K}_t =\mathcal{K}_t(\eta): =\left\{\phi\in\mathcal{H}_t : \log\max_{1\le i\le n}|\alpha_{\phi,i}(p)|\le 1/(\log_3t)(\log_4t),\forall p\le (\log t)^{1/\eta}\right\}
$$
where $\eta>0$ is any number. The size of the exceptional set, i.e. $\mathcal{H}_t^-= \mathcal{H}_t\backslash \mathcal{K}_t$, is small:
\begin{equation}\label{exceptional number}
\mathcal{H}_t^- \ll t^d\exp\left(-c \frac{\eta \log t}{(\log_2t)(\log_3t)(\log_4t)}\right)
\end{equation}
for some constant $c>0$,  by \cite[Theorem 7.3]{lw} (see also \cite[(6.1)]{lw}).  We work on $\mathcal{K}_t$ with the argument in \cite{lwu} to complete the proof.
\end{proof}

Now we prove Theorem \ref{M-V conjecture}.
For $\phi\in \mathcal{H}_t\backslash E^1(z)$, we have
\begin{align*}
&|L(1,\phi)|\le\left\{1+O\left(\frac{1}{\log_2 t}\right)\right\}\prod\limits_{p\le z}\left(1-\frac{\alpha'}{p}\right)^{-n}\\
&\le\left\{1+O\left(\frac{1}{\log_2 t}\right)\right\}(e^\gamma\log z)^{\alpha'n}\\
&\le\left\{e^\gamma\bigg((e^{\gamma(1-1/\alpha')}\log z)^{\alpha'}+C_0(\log_2 t)^{\alpha'-1}\bigg)\right\}^n,
\end{align*}
where $C_0$ is an absolute constant and $\alpha'=\exp(1/(\log_3t)(\log_4t))$. Taking
\begin{align*}
z
&=e^{e^{-\gamma(1-1/\alpha')}(\log_2t+r-C_0(\log_2 t)^{\alpha'-1})^{1/\alpha'}}\\
&=e^{(1+O((\log_4t)^{-1})(\log_2t+r-C_0(\log_2 t)^{\alpha'-1})},
\end{align*}
the proof is complete for $F_t^+$. The case of $F_t^-$ is treated in the same fashion.

\vskip5mm
\noindent\textbf{Acknowledgments}. We would like to thank the referees for their reading and comments.


\emph{
Lau is supported by GRF (Project Code. 17302514 and 17305617) of the Research Grants Council of Hong Kong. Wang is supported by National Natural Science Foundation of China (Grant No. 11871344).}

\end{document}